\documentclass{birkjour}
\usepackage{bm,comment}

\newtheorem{theorem}{Theorem}[section]
\newtheorem{lemma}[theorem]{Lemma}

\newtheorem{condition}[theorem]{Condition}

\theoremstyle{definition}
\newtheorem{definition}[theorem]{Definition}
\newtheorem{example}[theorem]{Example}

\theoremstyle{remark}
\newtheorem{remark}[theorem]{Remark}

\newcommand{\R}{\mathbb{R}}
\renewcommand{\div}{{\rm div}}
\newcommand{\curl}{{\rm curl}}
\newcommand{\grad}{{\rm grad}}
\newcommand{\norm}[1]{{\left\|#1\right\|}}

\numberwithin{equation}{section}

\begin{document}

\title{Unique Measure for Time-Dependent Random Dynamical Systems}

\author{Gregory Varner}
\address{Division of Natural Sciences, Mathematics Department
John Brown University, Siloam Springs, AR 72761}
\email{gvarner@jbu.edu}


\date{Received: date / Accepted: date}


\begin{abstract}
This paper proves the uniqueness of measure for the two-dimensional Navier-Stokes equations under a random kick-force and a time-dependent deterministic force. By extending a result for uniqueness of measure for time-homogeneous Markov processes to the time-inhomogeneous case, it is shown that the measures are exponentially mixing for the 2D Navier-Stokes equations on the sphere.
\end{abstract}

\maketitle

\section*{Introduction}

The existence and uniqueness of measure for the Navier-Stokes equations has been the subject of much recent research, with the main focus being unique time-invariant measure. A major advance was achieved in \cite{Kuksin2} where it was shown that, under a random bounded kick-type force, the Navier-Stokes system on the torus (bounded domains with smooth boundaries and periodic boundary conditions) has a unique time-invariant measure. In subsequent publications (see \cite{Kuksin,Kuksin3,Kuksin4,Shirikyan}) the argument was refined to a more flexible coupling approach, which has extended the argument to white noise case, time-periodic cases (see \cite{Shirikyan4,Varner2}) and the equations on the sphere (see \cite{Varner,Varner2}). Unfortunately, for true meteorological considerations it is also necessary to consider the equations under a fully time-dependent deterministic forcing. While this does, in general, negate the question of existence of time-independent measure, the question of uniqueness and the rate of convergence is still of interest.  

In this paper the previous results mentioned are extended to include time-dependent deterministic forces. In order to keep the presentation simple and to highlight similarities between the time-dependent case and previous work, we use a random perturbation activated by a dirac function as in \cite{Kuksin,Varner2} instead of a random perturbation activated by an indicator function as in \cite{Shirikyan4}. In addition, we include a more general case of a squeezing-type property of the deterministic equations similar to one used in \cite{Varner2}, which can allow for more general time-dependent deterministic forces.

The first section presents the main result of the paper, proving a theorem that extends the coupling argument in \cite{Kuksin7} to time-dependent forces. In particular, it is shown that for an inhomogeneous Markov process that has nonzero probability of coming arbitrarily close together in finite time (see Condition \ref{controllability}) and has a coupling between each time step that has a positive probability of being half the initial distance apart (see Condition \ref{thisisacouplinglemma}) then the associated probability measures are exponentially mixing (converge exponentially). Thus, regardless for any initial distribution, there is only one limiting measure. It should be mentioned that different but equivalent conditions are used to show uniqueness of measure in \cite{Shirikyan4}. The conditions chosen here are to highlight the dependence on the behavior of the underlying deterministic system. 

The second section presents an application of the main result to the Navier-Stokes equations on the sphere. A combination of the approaches in \cite{Brzezniak,Ilin,Ilin5,Titi} are used to first define the deterministic Navier-Stokes equations in the Navier-Stokes equations on the sphere, giving special attention a couple of specific examples of deterministic forcing that guarantee that Condition \ref{controllability} holds. The perturbed equations are then described and conditions are given that guarantee Condition \ref{controllability} without the necessity of a globally attracting solution are presented. Finally, necessary conditions for Condition \ref{thisisacouplinglemma} to hold are presented.


\section{The Main Theorem}

\subsection{Preliminaries}
Let $\left(u_{t},\mathbb{P}_{u}\right), \ t\geq 0$, be a Feller family of (time-inhomogeneous) Markov processes defined on a measurable space $\left(\Omega,\mathcal{F}\right)$ with range in $H$, a separable Hilbert space. Let $\beta\left(t,u,\Gamma\right)$ be the transition function associated with $\left(u_{t},\mathbb{P}_{u}\right)$,
$$P(t,u,\Gamma)=\mathbb{P}_{u}\left\{u_{t}\in\Gamma\right\}, \ u\in H, \Gamma\in \mathcal{B}(H).$$

Let $\beta_{t}$ and $\beta_{t}^{*}$ be the corresponding Markov operators
$$\beta_{t}f(u)=\int_{H}f(v)\beta(t,u,dv), \quad \beta_{t}^{*}\mu(\Gamma)=\int_{H}\beta(t,v,\Gamma)\mu(dv),$$
where $f\in C_{b}(X)$, $\mu\in \mathcal{P}(H)$, and $\Gamma\in \mathcal{B}(H)$. In order to examine time-inhomogeneous Markov processes, we will need the following operators which consider the behavior of the Markov process between two fixed times
$$\beta(t,l,u,\Gamma) = \mathbb{P}_{u}\left\{u_{t}\in\Gamma| u_{l}=u\right\}$$
and the corresponding Markov operators
$$\beta_{t,l}f(u)=\int_{H}f(v)\beta(t,l,u,dv), \quad \beta_{t,l}^{*}\mu(\Gamma)=\int_{H}\beta(t,l,u,\Gamma)\mu(du).$$

As mentioned, this operator considers the flow from time $l$ to time $t$ given the state at time $l$. Of course, if the Markov process if time-independent then $\beta(t,l,v,\Gamma)=\beta(t-l,v,\Gamma)$ since the behavior only depends on the time-elapsed.

Note that $\beta(t,l,u,\Gamma)$ satisfies the Chapman-Kolmogorov relation 
$$\beta(t,l,u,\Gamma)=\int_{X}\beta(s,l,z,\Gamma)\beta(t,s,u,dz)$$
which implies that for $l\leq s \leq t$
$$\beta_{t,l}f(v) = \beta_{t,s}\circ \beta_{s,l}f(v) \quad \text{and} \quad \beta_{t,l}^{*}\mu(\Gamma) = \beta_{t,s}^{*}\circ \beta_{s,l}^{*}\mu(\Gamma).$$
Thus both $\beta_{t,l}$ and $\beta_{t,l}^{*}$ form semi-groups in their respective spaces. Furthermore, this means that for any positive integer $t$
$$\beta_{t}^{*}\mu(\Gamma)=\beta_{t,0}^{*}\mu(\Gamma) = \beta_{t,t-1}^{*}\circ \beta_{t-1,t-2}^{*}\circ \cdots \beta_{1,0}^{*}\mu(\Gamma).$$

If the deterministic force is time-periodic or time-independent (and thus periodic) as considered in \cite{Kuksin,Shirikyan4,Varner2} then the semi-group property becomes $\beta_{t}=\beta_{1}\circ\beta_{1}\circ\cdots\circ \beta_{1}$.

\begin{example}\label{thisishownavierstokesworks}
Let $H$ be a separable Hilbert space, for each fixed $k\geq 1$ let $S_{t,t-1}:H\rightarrow H$ be a continuous mapping, and $\left\{\eta_{k}\right\}$ be a sequence of independent and identically distributed random variables in $H$ defined on a complete probability space $\left(\Omega,\mathcal{F},\mathbb{P}\right)$. Fixing an initial $v\in H$, consider the sequence of random variables given by the rule:
\begin{equation}
\begin{split}
	& u^{0}(v)= v \\
  & u^{k+1}(v)
  = S_{k+1,k}u^{k}(v)+\eta_{k+1}(x),  \ k=0,1,2,\ldots \\
  & u^{k+\tau}(v_{0}) = S_{k+\tau,k}u^{k}(v), 
   \ 0\leq\tau < 1, 
  \ k=0,1,2,\ldots
\end{split}
\end{equation}
In other words, in between the ``kicks'' the equations are governed by the (time-dependent) continuous mapping.

This defines a time-inhomogeneous discrete-time Markov process in $H$ (similar to the analogous formula in \cite{Kuksin7}, p. 24). Furthermore, notice that $\beta(k,k-1,v,\Gamma)$ is the probability that $S_{k,k-1}v+\eta_{k} \in \Gamma$. Thus it is the ``one-step'' Markov transition function for the process.
\end{example} 

For the statement of the main theorem, it will be necessary to recall the concept of a {\it coupling} of two measures. A pair of random variables $\left(\zeta_{1},\zeta_{2}\right)$ defined on a probability space $\Omega$ is called a coupling for given measures $\mu_{1}, \mu_{2}$ if the distribution $\mathcal{\zeta_{j}}=\mu_{j}$, $j=1,2.$

\subsection{The Main Result}

Let $\left(u_{t},\mathbb{P}_{u}\right), \ t\geq 0$, be a Feller family of (time-inhomogeneous) Markov processes defined on a measurable space $\left(\Omega,\mathcal{F}\right)$ with range in $H$, a separable Hilbert space with norm $\norm{\cdot}_{H}$. Let $\beta\left(t,u,\Gamma\right)$ and $\beta(t,t-1,u,\Gamma)$ be the transition functions associated with $\left(u_{t},\mathbb{P}_{u}\right)$ as described above. 

Suppose that the following two properties are satisfied:

\begin{condition}\label{controllability}
For any $d>0$ and $R>0$ there exists integer $l=l(d,R)>0$ and real number $x=x(d)>0$ such that
        \begin{equation}
 \mathbb{P}\left\{\|u^{l}(v)-u^{l}(w)\|_{H}\leq d\right\} \geq x, \ \text{for \ all} \ v,w\in B_{H}(R),
        \end{equation}
        where $B_{H}(R)$ is the ball of radius $R$ centered at 0 in $H$.
\end{condition}

\noindent Condition \ref{controllability} gives that there is a positive probability that the perturbed flow will becomes arbitrarily close together in finite time regardless of the initial conditions and, thus, provides a controllability assumption on the perturbed flow.

Denote by $X^{k}$ the direct product $X\times\cdots\times X$ endowed with the $\sigma$-algebra $\mathcal{B}^{k}(X)=\mathcal{B}(X)\times \cdots \times \mathcal{B}(X)$.

\begin{condition}\label{thisisacouplinglemma}
For any $R>0$ and for each fixed $t\geq 0$ there exists a constant $d>0$ such that for any points $u,u'\in B_{H}(R)$ with $\norm{u-u'}_{H}\leq d$ the measures $\beta(t,t-1,u_{1,2},\cdot)$ admit a coupling $V^{t,t-1}_{1,2}=V^{t,t-1}_{1,2}(u_{1},u_{2};\omega)$ that is measurable with respect to $\left(\vec{u}_{1},\vec{u}_{2},\omega\right)\in B_{H}(R)^{2}\times\Omega$ such that 
\begin{equation}
\mathbb{P}\left\{\norm{V^{t,t-1}_{1}-V^{t,t-1}_{2}}_{H}\geq \frac{d}{2}\right\}\leq Cd \nonumber
\end{equation}
where $C>0$ does not depend on $u,u',$ or $t$.
\end{condition}

Condition \ref{thisisacouplinglemma}, which in application follows from properties of the underlying deterministic system, will allow the proof of the main result to be reformed as a result about random variables in $H$ and says that there is a positive probability that the distance between the random variables is converging to zero. Furthermore, note that the condition only has significance if $Cd<1$. Due to this, define $d_{0}$ such that
$$Cd_{0}<\frac{1}{32}.$$

\begin{theorem}\label{themaintheoremhere}
Let $R>0$ and suppose that Conditions \ref{controllability} and \ref{thisisacouplinglemma} hold, then for any $u,v\in B_{H}(R)$ 
\begin{equation}
\norm{\beta(k,u,\cdot)-\beta(k,v,\cdot)}_{L}^{*}\leq C_{R}e^{-ck}, \quad k\geq 1,
\end{equation}
where $c>0$ is a constant not depending on $R$.
\end{theorem}

In particular, regardless of the initial measure (conditions), measures are exponentially mixing.

\subsection{Proof of the Main Theorem}

The structure of the proof of Theorem \ref{themaintheoremhere} is very similar to the arguments found in \cite{Kuksin} and \cite{Kuksin8} and contain many of the same elements as the analogous results for time-independent and time-periodic deterministic forces (c.f. Theorem 3.2.5 in \cite{Kuksin7}, Theorem 2.5 in \cite{Shirikyan4}, or Theorem 25 in \cite{Varner2}). However, despite the similarities in the method and structure of the proof, the application of the argument to time-inhomogeneous Markov chains is a significant conceptual departure from previous uses. 

The main portion of the proof follows from the behavior of the coupling. In particular, it follows from Condition \ref{thisisacouplinglemma} which gives that when the initial conditions are ``close enough" together, there is a positive probability that the distance between the coupling for the transition functions will be half the original distance. This idea is given more precisely in the following lemma.

Recall that $d_{0}$ is such that $Cd_{0}<\frac{1}{32}$ where $C$ is from Condition \ref{thisisacouplinglemma}.

\begin{lemma}
For any $R>0$ let $\vec{u}_{1,2}\in B_{H}(R)$ and $d=\norm{\vec{u}_{1}-\vec{u}_{2}}$. Then under the condition of Theorem \ref{themaintheoremhere}, there is a probability space $\left(\Omega,\mathcal{F},\mathbb{P}\right)$ such that for any integer $k\geq 1$ the measures $\mu_{\vec{u}_{1,2}}(k)$ admit a coupling $\vec{U}_{1,2}^{k}= \vec{U}_{1,2}^{k}(\vec{u}_{1},\vec{u}_{2},\omega^{k}), \ \omega^{k}\in\Omega^{k}$ such that
	\begin{enumerate}
		\item The maps $\vec{U}_{1,2}^{k}$ are measurable with respect to $\left(\vec{u}_{1},\vec{u}_{2},\omega^{k}\right)\in B_{H}(R)^{2}\times\Omega^{k}$.

		\item If $d=\norm{\vec{u}_{1}-\vec{u}_{2}}\leq 2^{-r}d_{0}$ then
			\begin{equation}
				\mathbb{P}^{k}\left\{\norm{\vec{U}_{1}^{k}-\vec{U}_{2}^{k}}\leq 2^{-k-r}d_{0}\right\}\geq 1-2^{-r-3}, \ k\geq 1, \ r\geq 0.    \label{singlecouplingestimate}
			\end{equation}
		\end{enumerate}
\end{lemma}

\begin{proof}
Since many of the calculations are identical to the proof of Lemma 3.3 in \cite{Kuksin}, we will highlight the differences by constructing the coupling operator. Recall that for $u_{1}, \ u_{2} \in B_{H}(R)$ for any fixed $k\geq 1$ a coupling $V^{k,k-1}_{1,2}(u_{1},u_{2};\omega)$ exists by assumption. For $j=1,2$ define
\begin{equation}
	 U^{k,k-1}_{j}(u_{1},u_{2};\omega):= 
			\left\{ \begin{array}{ll}
				V^{k,k-1}_{j}(u_{1},u_{2};\omega),& \mbox{if} \ \norm{u_{1}-u_{2}}\leq d_{0}\\    
				u^{k,k-1}(u_{j}), & \text{if} \ \norm{u_{1}-u_{2}}>d_{0}\\
				\end{array} \right. .
\end{equation}
In other words, if $\norm{u_{1}-u_{2}}\leq d_{0}$ then we proceed by using the coupling and if $\norm{u_{1}-u_{2}}> d_{0}$ we let the random process to continue. Of course, by Condition \ref{controllability} there is a finite time $l(d_{0})$ such that there is a positive time that the random process will be within $d_{0}$. 

 With this in mind, the random variables $U^{k}_{1,2}$ are defined inductively on $\left(\Omega^{k},\mathcal{F}^{k}\right)$ as follows:

\begin{eqnarray}
		&	U^{k}_{j}(u_{1},u_{2};\omega^{k}) :=   \\
		&	\left\{\begin{array}{ll}
			u^{k}(u_{j})
					& \mbox{if} \ \norm{u_{1}-u_{2}}>d_{0}, \\
					\ & \ k\leq l(d_{0}) \\
				U^{k,k-1}_{j}(U^{k-1}_{1}(u_{1},u_{2};\omega^{k-1}),U^{k-1}_{2}(u_{1},u_{2};\omega^{k-1}),
				\omega_{k}) & \text{if} \ \norm{u_{1}-u_{2}}>d_{0}, \\
				\ & \  k>l(d_{0}) \\
				U^{k,k-1}_{j}(U^{k-1}_{1}(u_{1},u_{2};\omega^{k-1}),U^{k-1}_{2}(u_{1},u_{2};\omega^{k-1}),
				\omega_{k}) & \text{if} \ \norm{u_{1}-u_{2}}\leq d_{0}
				\end{array} \right.   \nonumber\label{uk}
\end{eqnarray}
where
\begin{eqnarray}
	& \omega^{k}= \left(\omega^{k-1},\omega_{k}\right), \nonumber\\
	& \text{and} \\
	& U_{j}^{0}(u_{1},u_{2}) = u_{j}.\nonumber 
\end{eqnarray}

The remainder of the proof follows the calculation on pp. 361-362 in \cite{Kuksin} with the exception that our choice of $d_{0}$ is $2^{-2}$ times the one used there.
\end{proof}

The remainder of the proof of Theorem \ref{themaintheoremhere} follows the identical calculation for Theorem 1.1 in \cite{Kuksin8} as the argument only requires that the above lemma holds.

\section{The Kicked Navier-Stokes Equations on the Sphere}

In this section we consider the 2D Navier-Stokes equations on the sphere under a kick-force, examining under what conditions the system will satisfy the requirements of Theorem \ref{themaintheoremhere}. However, in order to do so, it will be necessary to describe the deterministic Navier-Stokes equations on the sphere.

\subsection{The Deterministic Equations on the Sphere}

Let $S^{2}$ be the 2-dimensional sphere with coordinates $\lambda$, $0\leq \lambda \leq 2\pi$ and $\phi$, $-\frac{\pi}{2}\leq \phi \leq \frac{\pi}{2}$ (the geographical latitude). The Navier-Stokes equations on the rotating sphere $S^{2}$ are \cite{Ilin5}

\begin{equation}\label{deterministicnavierstokes}
	\begin{split}
		&	\partial_{t}\vec{u}
			+
			\nabla_{\vec{u}}\vec{u}
			-
			\nu\Delta \vec{u}
			+
			l\,\vec{n}\times \vec{u}
			+
			\nabla \ p 
			=
			\vec{f},  \\
		& \div \ \vec{u} =0, 	\ \vec{u}|_{t=0}=\vec{u}_{0},
	\end{split}
\end{equation}
where $\vec{u}$ is the tangent velocity vector, $p$ is the pressure, $f$ is the forcing terms, $\vec{n}$ is the unit outward normal vector, $l=2\Omega\sin\phi$ is the Coriolis coefficient, $\Omega$ is the angular velocity of the Earth, and ``$\times$" is the standard cross product in $\R^{3}$.

The operators $\div$ and $\nabla$ in Equation \eqref{deterministicnavierstokes} have their conventional meanings on the sphere, i.e. for functions $\psi$ and vectors $\vec{u}$

\begin{equation}
	\nabla \psi = \frac{\partial \psi}{\partial \phi}\vec{\phi}+\left(\frac{1}{\cos\phi}\frac{\partial \psi}{\partial \lambda}\right)\vec{\lambda},
	\quad
	\div\vec{u} = \frac{1}{\cos\phi}\left(\frac{\partial}{\partial \lambda}u_{\lambda}+\frac{\partial}{\partial \phi}\left(u_{\phi}\cos\phi\right)\right), \nonumber
\end{equation}
\noindent where $\vec{u}=u_{\lambda}\vec{\lambda} +u_{\phi}\vec{\phi}$.

In order to define the covariant derivative $\nabla_{\vec{u}}$ and the vector Laplacian $\Delta$ the following definitions will be needed (\cite{Ilin5}, p. 984):

\begin{definition} For a tangent vector $\vec{u}$ and a normal vector $\vec{\psi}=\vec{n}\psi$ (identifying the vector with the function)
$$\curl\,\psi = -\vec{n}\times\nabla\psi, \quad \curl_{n}\vec{u}=-\vec{n}\div\left(\vec{n}\times \vec{u}\right).$$
\end{definition} 

With these definitions, we have that for a tangent vector $\vec{u}$ (\cite{Ilin5},p. 984)
\begin{eqnarray}
	\nabla_{\vec{u}}\vec{u} &:= \nabla\frac{\left|\vec{u}\right|^{2}}{2}
	-
	\vec{u}\times \curl_{n}\vec{u}, \\
	\Delta \vec{u} & := \nabla\,\div\, \vec{u} - \curl\,\curl_{n}\vec{u}.
\end{eqnarray}

It is worth noting that $\curl_{n}$ maps tangent vectors to normal vectors and $\curl$ maps normal vectors to tangent vectors. Furthermore, for a normal vector $\psi$ 
\begin{equation}
\curl_{n}\,\curl \psi = -\vec{n}\Delta \psi
\end{equation}
where $\Delta$ is the scalar spherical Laplacian.

\begin{remark} The above operators can be defined through extensions. In particular, for any covering $\left\{O_{i}\right\}$ of $S^{2}$ by open sets, there is a corresponding set of ``cylindrical domains" $\widetilde{O}_{i}$ that cover a tubular neighborhood of $S^{2}$, $\widetilde{S}^{2}$. In each $\widetilde{O}_{i}$ introduce the orthogonal coordinate system $\widetilde{x}_{1}, \ \widetilde{x}_{2}, \ \widetilde{x}_{3}$, where $-\epsilon<\widetilde{x}_{3}< \epsilon$ is along the normal to $S^{2}$ and for $\widetilde{x}_{3}=0$ the coordinates $x_{1},x_{2}$ agree with the spherical coordinates. Each of the operators area then defined as restrictions back onto the sphere (see \cite{Brzezniak,Ilin,Titi}).
\end{remark}

Now let $L^{p}(S^{2})$ denote the standard $L^{p}$-spaces of the square integrable scalar functions $\psi$ with mean value zero and tangent vector field $\vec{u}$ on $S^{2}$ with norms
 $$\norm{\psi}_{L^{2}}:= \int_{S^{2}}\psi^{2} dS^{2} $$
$$\norm{\vec{u}}_{L^{2}}:= \int_{S^{2}}\vec{u}\cdot \vec{u} dS^{2}.$$
	
\noindent Note that while these are integrals over oriented manifolds, locally $dS^{2}=\cos\phi d\phi d\lambda$. 

\noindent Let $\psi$ be a scalar function and $\vec{v}$ be a vector field on $S^{2}$. For $s\geq 0$, the standard Sobolev spaces $H^{s}$ have norm 
$$	\norm{\psi}^{2}_{H^{s}}:= \norm{\psi}^{2}_{L^{2}}
	+
	\left\langle -\Delta^{s} \psi, \psi\right\rangle_{L^{2}}$$
and

$$	\norm{\vec{u}}^{2}_{H^{s}}:= \norm{\vec{u}}^{2}_{L^{2}}
	+
	\left\langle -\Delta^{s} \vec{u}, \vec{u}\right\rangle_{L^{2}}.$$

\noindent By the Hodge Decomposition Theorem, the space of smooth vector fields on $S^{2}$ can be decomposed as (\cite{Ilin}, p. 564):
\begin{eqnarray}
		C^{\infty}(S^{2}) \nonumber \\
		= &\left\{\vec{u}:\vec{u}=\grad\phi, \phi\in C^{\infty}(S^{2})\right\}\oplus
									\left\{\vec{u}: \vec{u}=\curl\phi, \phi\in C^{\infty}(S^{2})\right\} \nonumber \\
	 = &\left\{\vec{u}: \vec{u}=\grad\phi, \phi\in C^{\infty}(S^{2})\right\}\oplus
									V_{0}. \nonumber						
\end{eqnarray}

\begin{definition}  
Let $H:= \curl(H^{1}(S^{2}))$ and $V:=\curl(H^{2}(S^{2}))$, which are closed subspaces of $L_{2}(S^{2})$ and $H^{1}(S^{2})$ respectively.
\end{definition}

Note that $H$ is the $L^{2}$ closure of $V_{0}$ and thus $\div\,\vec{u}=0$ for $\vec{u}\in H$ and $V$ is the $H^{1}$ closure of $V_{0}$ and thus $\div\,\vec{u}=0$ for $\vec{u}\in V$. Furthermore, $V$ is compactly embedded into $H$, and by the Poincare Inequality the $V$ norm is equivalent to the $H^{1}$ norm for divergence-free vector fields (see \cite{Ilin}, pp. 563-565).

Since the equations will be defined on spaces of divergence-free vector fields, the following definition will be useful.

\begin{definition} For a vector field $\vec{u}$, define the Laplacian on divergence-free vector fields as
	\begin{equation}
		A\vec{u}:= \curl\curl_{n}\vec{u}.
	\end{equation}
	Furthermore, if $\div\,\vec{u}=0$ then $A\vec{u}= -\Delta \vec{u}$. 
\end{definition}

Since the operator $A=\curl\curl_{n}$ is a self-adjoint positive-definite operator in $H$ it has eigenvalues $0< \lambda_{1}\leq \lambda_{2}\leq ... $ with the only accumulation point $\infty$ that correspond to an orthonormal basis in $H$ and an orthogonal basis in $V$. 

Let $P_{H}$ be the projection onto $H$. Since the projection commutes with $\partial_{t}$ and $A$, the projection of the Navier-Stokes equations onto $H$ is 
\begin{equation}
	\partial_{t}\vec{u}
	+
	\nu A\vec{u}
	+ 
	B(\vec{u},\vec{u}) 
	+
	C(\vec{u})
	= 
	\vec{f}, \quad \vec{u}|_{t=0}=\vec{u}_{0}, \label{projectednavierstokes}
\end{equation}
where $B(\vec{u},\vec{u})+C(\vec{u})=P_{H}(\nabla_{\vec{u}}\vec{u}+ l\,\vec{n}\times \vec{u})$. Furthermore, for all $\vec{v}\in V$
\begin{equation}
 	\left\langle B(\vec{u},\vec{u}) + C(\vec{u}),\vec{v}\right\rangle_{H} 
 	= 
 	b(\vec{u},\vec{u},\vec{v}) + \left\langle l\,\vec{n}\times \vec{u},\vec{v}\right\rangle_{H}, 
\end{equation}
where $b(\vec{u},\vec{u},\vec{v})= \left\langle -\vec{u}\times \curl_{n}\vec{u},\vec{v} \right\rangle$ and in general (\cite{Brzezniak}, p. 8)
$$b(\vec{u},\vec{v},\vec{w})=\frac{1}{2}\int_{S^{2}}\left(-\vec{u}\times\vec{v}\cdot \curl_{n}\vec{w}+\curl_{n}\vec{u}\times\vec{v}\cdot\vec{w}-\vec{u}\times\curl_{n}\vec{v}\cdot\vec{w}\right)dS^{2}.$$

\begin{remark}
$b(u,v,w)$ is the standard trilinear form associated with the Navier-Stokes equations, i.e. 
\begin{equation}
	b(\vec{u},\vec{v},\vec{w}) = \pi \sum_{i,j=1}^{3}\int_{M}u_{j}D_{i}v_{j}w_{j}dx,
\end{equation}
where $\pi$ is the orthogonal projection onto $S^{2}$ and the trilinear terms satisfies estimates analogous to those in the case of flat domains (\cite{Ilin}, pp. 561-566).
\end{remark}


We now state the existence and uniqueness of solutions to the deterministic Navier-Stokes equations in terms of the projected equations, as is standard. The proof is the same as the case of bounded domains with smooth boundaries and periodic boundary conditions (see \cite{Robinson}, pp. 245-254 and \cite{Ilin}, Theorem 2.2). 


\begin{theorem}\label{existenceofprojectedsolution} Suppose $\vec{f} \in  L^{2}(0,T; H)$ and $\vec{u}_{0}\in H$ then a solution of equation \eqref{projectednavierstokes} exists uniquely and $\vec{u}\in L^{2}(0,T;V)\cap C([0,T];H)$. If $\vec{u}_{0}\in V$ then the solution is strong, i.e. $\vec{u}\in L^{2}(0,T;D(A))\cap C([0,T];V)$ and $\dfrac{d\vec{u}}{dt} \in L^{2}(0,T; H)$. 
\end{theorem}

Exponentially stable solutions are of special interest in the following section, where {\bf exponential stability} means that for any $R>0$, for all $t\geq t_{0}$, and for some $\alpha>0$
\begin{equation} \label{FForceCondition}
    \norm{S_{t}\vec{u}_{0}-S_{t}\vec{v}_{0}}_{H}
    \leq
    C(R)e^{-\alpha (t-t_{0})}
    \norm{\vec{u}_{0}-\vec{v}_{0}}_{H}
    \ \forall \ \vec{u}_{0},\vec{v}_{0}\in B_{H}(R)
\end{equation}
where $C(R)$ can depend on the norm of the force, $u(t_{0})=u_{0}$, and $S_{t}$ is the solution operator for the Navier-Stokes equations, i.e. $S_{t}u = u(t)$.

\begin{lemma}\label{whenhaveexponentialconvergence} The solution is exponentially stable if one of the following conditions holds:
\begin{enumerate}
\item 
If 
\begin{equation} 
    \norm{\vec{f}}_{L^{\infty}(0,\infty;H)} < \dfrac{\nu^{2}\sqrt{\lambda_{1}}}{k}. \label{thisishowsmall}
\end{equation}
\item
If the force yields a solution of the form $g(t)\curl \sin(\phi)\vec{n}$.
\item If the force is within $\delta$ of a force that yields a solution of the form $g(t)\curl \sin(\phi)\vec{n}$, for a particular choice of delta. (We call such a solution {\bf almost zonal}.)
\end{enumerate}
\end{lemma}

Since the proof of Lemma \ref{whenhaveexponentialconvergence} is technical in nature, the proof is reserved for the appendix.

\begin{remark} Due to \cite{Heywood}, p.19, if the deterministic force is time-periodic (or time-independent), then the existence of an exponentially stable solution guarantees that the solution will have the same period as the force. 
\end{remark}


\subsection{The Perturbed Equations on the Sphere}

We now turn to the Navier-Stokes equations on the sphere perturbed by a random kick-type force. The perturbations considered are an external forcing on the system activated by a Dirac function and in between the perturbations the system is governed by the deterministic equations. In particular, consider the Navier-Stokes system with forcing $\vec{f}\in L^{\infty}\left(0,\infty;H\right)$ and a random kick-force $\vec{g}$ bounded in $H$:
\begin{equation} \label{kicked}
\begin{split}
	&\partial_{t}\vec{u}
	+
	\nu A\vec{u}
	+
	B(\vec{u},\vec{u})
	+
	C(\vec{u})
	=
	\vec{f}+ \vec{g},\\    
	&\vec{g} = 
	\sum_{k=1}^{\infty}
	\eta_{k}(x)\delta_{kT}(t),
	 \ \vec{\eta}_{k}\in H, 
	 \  \norm{\vec{\eta}_{k}}_{H}< \infty \ \forall k.
\end{split}
\end{equation}

Furthermore, as in \cite{Kuksin}, pp. 356-357, we assume the kick-force satisfies:

\begin{condition} \label{mainkicksassumption}
Let $\left\{\vec{e}_{j}\right\}$ be the orthonormal basis for the Hilbert space $H$, 
then 
		\begin{equation} 
		 	\vec{\eta}_{k} = \sum_{j=1}^{\infty}b_{j}\zeta_{jk}\vec{e}_{j}, \ b_{j}\geq 0, \ B_{0}=
		 	 \sum_{j=1}^{\infty}b_{j}^{2}<\infty, 
		\end{equation}
for $\left\{\zeta_{jk}\right\}$ a family of independent, identically distributed real-valued variables, with $\left|\zeta_{jk}\right|\leq 1$ for all $j,k$. 
Their common law has density $p_{j}$ with respect to Lebesgue measure where $p_{j}$ is of bounded variation with support in the interval $\left[-1,1\right]$.
Furthermore, for any $\epsilon>0$, $\displaystyle \int_{\left|r\right|<\epsilon}p_{j}(r)dr>0.$  
\end{condition}

\noindent The notation from now on will be:
\begin{itemize}
	\item $S_{t,s}\vec{v}_{0}$ is the solution of the deterministic equation with initial condition at time $t=s$ of $\vec{v}_{0}\in H$ at time $t\geq s$.
	\item $S_{t,0}\vec{v}_{0}=S_{t}\vec{v}_{0}$.
	\item For simplicity of notation take the time between kicks as $T=1$. 
	\item $\vec{u}^{t}(\vec{v}_{0})$ is the solution of \eqref{kicked} with initial condition $\vec{v}_{0}$ at time $t\geq 0$. 
\end{itemize}

Then, the perturbed system is described by
\begin{equation}
\begin{split}
	& u^{0}(v)= v \\
  & u^{k+1}(v)
  = S_{k+1,k}u^{k}(v)+\eta_{k+1}(x),  \ k=0,1,2,\ldots \\
  & u^{k+\tau}(v_{0}) = S_{k+\tau,k}u^{k}(v), 
   \ 0\leq\tau < 1, 
  \ k=0,1,2,\ldots
\end{split}
\end{equation}
It is worth noting that if all of the ``kicks'' were the zero realization, then $S_{n}\vec{v}_{0}= \vec{u}^{n}(\vec{v}_{0}).$

Since the perturbed system is of the same form as in Example \ref{thisishownavierstokesworks}, it remains to examine when Conditions \ref{controllability} and \ref{thisisacouplinglemma} hold. While it is fairly straightforward that the existence of an exponentially stable solution is sufficient for Condition \ref{controllability}, it also guarantees that any two deterministic solutions with different initial conditions will become arbitrarily close together as $t\rightarrow \infty$. Similarly, any point that acts locally like an asymptotically stable solution will be a (local) contraction of the flow and needs to be considered. Since it will be sufficient for the random perturbations to be finite-dimensional (see Theorem \ref{THEMAINTHEOREM}), it will suffice for the solution to be locally stable in a finite number of dimensions. It should be pointed out that the following definition also captures the same concept as determining modes (\cite{Robinson}, p. 363).

\begin{definition}\label{finitelystable} Let $D(f)$ be the radius of the deterministic absorbing set (see \cite{Ilin}, p. 572) and $P_{M}$ be the projection onto the first $M$ eigenfunctions. A point $\vec{u}\in B_{H}(D(f))$ is called {\bf finitely stable} if for some $M\geq 1$, for some $\delta>0$, for any $\epsilon>0$ and for all $\vec{v}\in B_{H}(D(f)+\epsilon)$ satisfying $\norm{P_{M}\vec{u}-P_{M}\vec{v}}_{H}\leq \delta$, 
\begin{equation}
	\norm{S_{t}\vec{u}-S_{t}(\vec{v})}_{H}\rightarrow 0. 
\end{equation}
In other words, if the finite-dimensional projections are ``close enough,'' then the solutions converge.
\end{definition}

Though the assumption of a finitely stable point allows for the possibility of multiple solutions, it also will require additional assumptions for the structure of the perturbations.

\begin{definition} The following is called the {\bf big kick assumption.} 
Let $M$ be as in Definition \ref{finitelystable}. For some $N\geq M$ let the $b_{j}$ from Condition \ref{mainkicksassumption} satisfy
\begin{align}
	& b_{1}\geq 2D(f),\nonumber \\
	& b_{j}\geq \dfrac{2D}{\lambda_{j}^{1/2}} \ \text{for} \ 2\leq j \leq M, \\ 
	& b_{j}>0 \ \text{for} \ M<j\leq N . \nonumber
\end{align}
where $D=D(f)$ is the same as in \eqref{Condition1b} and $\lambda_{j}$ is the eigenvalue corresponding to $\vec{e}_{j}(x)$.
\end{definition}

By the Poincare inequality and standard estimates for the $H^{1}$ norm (c.f. Equation 4.8 in \cite{Ilin}) the $b_{j}$ are assumed to be twice as large as $\norm{Q_{j-1}\vec{u}(t)}_{H}$ if the initial condition is zero (where $Q_{n}=I-P_{n}$). In particular, if the stochastic flow is within $\delta$ of the ball of radius $D(f)$ then the kicks are large enough to ``kick'' the first $M$-dimensions of the flow within $\delta$ of the first $M$ dimensions of any point, thus a finitely stable point, in the deterministic absorbing ball with nonzero probability. In other words, the big kick assumption allows the perturbation to ``kick" the flow from anywhere in the absorbing ball into the stability radius of a finitely stable point.

\begin{theorem} \label{THEMAINTHEOREM}
Let the kicks satisfy Condition \eqref{mainkicksassumption}, let $\vec{f}\in L^{\infty}(0,\infty;H)$, and that either:
\begin{itemize}
	\item there exists at least one finitely-stable point and the
big kick assumption holds or 
	\item there is an exponentially stable solution.
\end{itemize}
Then there is an $N$ such that if
$b_{j}>0$ for $j=1,2,...,N$ then Conditions \ref{controllability} and \ref{thisisacouplinglemma} hold.
\end{theorem}

The proof of Theorem \ref{THEMAINTHEOREM} is based on the following two lemmas, which follow from showing that the 2D Navier-Stokes equations on the sphere satisfy the following conditions. These follow from standard estimates of the Navier-Stokes equations (included in the Appendix for completion). 

\begin{condition} \label{generalcontraction}
For any $R$ and $r$ with $R> r>0$ there exist $C= C(R,f),$ $D=D(f)$, $a= a(R,r)< 1$ all positive and there exists an integer $n_{0}=n_{0}(R,r)\geq 1$ such that
         \begin{align}
             & \norm{S_{n}\vec{u}_{0}}_{H} \leq \max\left\{a\norm{\vec{u}_{0}}_{H}+ D,r + D\right\}, \ \vec{u}_{0}\in B_{H}(R), \ \forall n\geq n_{0} \label{Condition1b}, \\
             & \norm{S_{t,t-1}\vec{u}_{0}-S_{t,t-1}\vec{v}_{0}}_{H} \leq C \norm{\vec{u}_{0}-\vec{v}_{0}}_{H},
              \ \forall t\in\left[1,\infty\right), \ \vec{u}_{0},\vec{v}_{0}\in B_{H}(R); \label{Condition1a}
         \end{align}
where $\norm{\vec{\eta}_{k}}^{2}\leq B_{0}$ for all $k$.

Furthermore, there is a decreasing sequence $\gamma_{N}(R,f)>0, \ \gamma_{N}\rightarrow 0$ as $N\rightarrow \infty$ such that for all $t\in\left[1,\infty\right), \ \vec{u}, \vec{v}\in B_{H}(R)$
        \begin{equation}
            \norm{(I-P_{N})(S_{t,t-1}\vec{u}_{0}-S_{t,t-1}\vec{v}_{0})}_{H} \leq \gamma_{N}(R,f)\norm{\vec{u}_{0}-\vec{v}_{0}}_{H}, \label{Condition2}
        \end{equation}
        where $P_{n}$ is the projection onto the first $N$ eigenfunctions $\vec{e}_{j}$.
\end{condition}

\begin{lemma}\label{firstlemmaweneed}
Under the assumptions of Theorem \ref{THEMAINTHEOREM} for any $R>0$ and for each fixed $t\geq 0$ there exists a constant $d>0$ such that for any points $u,u'\in B_{H}(R)$ with $\norm{u-u'}_{H}\leq d$ the measures $\beta(t,t-1,u_{1,2},\cdot)$ admit a coupling $V^{t,t-1}_{1,2}=V^{t,t-1}_{1,2}(u_{1},u_{2};\omega)$ that is measurable with respect to $\left(\vec{u}_{1},\vec{u}_{2},\omega\right)\in B_{H}(R)^{2}\times\Omega$ such that 
\begin{equation}
\mathbb{P}\left\{\norm{V^{t,t-1}_{1}-V^{t,t-1}_{2}}_{H}\geq \frac{d}{2}\right\}\leq Cd \nonumber
\end{equation}
where $C>0$ does not depend on $u,u',$ or $t$.
\end{lemma}

\begin{proof}
By equation \eqref{Condition1a}, $S_{t,t-1}:H\rightarrow H$ is a Lipschitz operator for any fixed $t \in \left[1,\infty\right)$, with Lipschitz constant depending only on the one time unit elapsed. Since equation \eqref{Condition2} is also satisfied for any fixed $t \in \left[1,\infty\right)$, the proof of Lemma \ref{firstlemmaweneed} follows identical to the proof of Lemma 3.2 in \cite{Kuksin}, which uses equations \eqref{Condition1a} and \eqref{Condition2}, with the exception that constants now depend on the norm of $f$. 
\end{proof}

The next lemma requires knowing that any sequence of realization of kicks can be taken, with positive probability, arbitrarily close to any other prescribed sequence of vectors in $supp D(\vec{\eta})$ (see \cite{Kuksin2}, Lemma 5.4). Furthermore, since the proof of the following lemma requires properties of the long term behavior of the deterministic system and the kicks but not the time-dependence of the force, the proof is identical to that of Lemmas 6 and 7 in \cite{Varner2} and thus only a brief sketch is given here.

\begin{lemma}
Under the assumptions of Theorem \ref{THEMAINTHEOREM} for any $d>0$ and $R>0$ there exists integer $l=l(d,R)>0$ and real number $x=x(d)>0$ such that
        \begin{equation}
 \mathbb{P}\left\{\|u^{l}(v)-u^{l}(w)\|_{H}\leq d\right\} \geq x, \ \text{for \ all} \ v,w\in B_{H}(R),
        \end{equation}
        where $B_{H}(R)$ is the ball of radius $R$ centered at 0 in $H$.
\end{lemma}

\begin{proof}
First suppose that all kicks are the zero realization. Then one of the following happen:
\begin{itemize}
	\item the system is exponentially stable and thus there is a time such that any two solutions become within $\delta/2$ or
	\item the system has a finitely stable point and there is a time such that any solution is within $d/2$ of the absorbing ball of radius $D(f)$, where $d$ is the stability radius of the finitely stable point.
\end{itemize}
The proof now proceeds as follows:
\begin{itemize}
	\item there is a positive probability that if all the kicks have size $\leq \gamma$ with $\gamma>0$ then the continuity of the system (equation \eqref{Condition1a}) implies the result.
	\item there exists a ``large kick" that can send the solution within $d/2$ of the finitely stable point. By the definition of a finitely stable point, there is a time such that the two solutions are within $\delta/2$. By continuity there is a $\gamma>0$ such that if the kicks are within $\gamma$ to the above prescribed sequence, then the result holds.
\end{itemize}
\end{proof}

Thus Theorem \ref{THEMAINTHEOREM} holds and Theorem \ref{themaintheoremhere} gives that there is an unique measure that any other initial distribution will converge to exponentially.


\section{Conclusion}

While there is uniqueness of measure for the kicked Navier-Stokes equations with a bounded time-dependent deterministic force, the requirements on the forces are rather restrictive. While zonal forces are applicable and important to meteorology, the forces are still restricted to the specific case of forcing along the first spherical harmonic. While the extension to the almost zonal solution does extend the results to more realistic atmospheric flows, the forces are still very restrictive in that they can only be a ``small'' distance from a zonal forcing in the first spherical harmonic. Thus, it would appear that the most applicable case to real atmospheric flows would be the finitely stable point. However, this requires a minimum size of kicks that are potentially unrealistic. Even though the probability of such a large kick occurring can be arbitrarily small by choice of the structure of the random variable, it must be nonzero and thus the ``large kick assumption'' is also problematic for meteorological considerations. 

It is possible, however, that the kicks may be allowed to be smaller. The big kick assumption is to ensure that a kick can, with positive probability, send the flow into the neighborhood of any point in the deterministic absorbing ball. The reason for the big kick assumption comes from the deterministic setting where a dirac measure at any stationary solution is a time-invariant measure, giving non uniqueness if there are multiple stationary solutions. Thus, for example, if there are two stable stationary solutions the kicks must be (at minimum) large enough to send the flow from inside the radius of stability of one into the radius of stability of the other. The big kick assumption is sufficient to do this, but a smaller, and thus more physically relevant, kick may suffice.

Of course, the majority of the results presented in this paper apply to the Navier-Stokes equations on the torus. While the results concerning the zonal solutions no longer hold, if the force yields an unique asymptotically stable solution or a finitely stable point then there is again unique measure for the time-dependent equations.

\section{Appendix}


\subsection{Estimates}

We now present estimates that will be needed to establish Condition \ref{generalcontraction} and Lemma \ref{whenhaveexponentialconvergence}. Since the majority of the estimates are well-known and analogous to standard estimates on flate domains, proof will only be provided for less standard results (cf. \cite{Constantin1, Ilin, Ilin2, Kuksin7}).

\begin{lemma} For $\vec{u}\in V$ the Poincare Inequality holds, i.e.
\begin{equation}
	\norm{\vec{u}}_{V}^{2}=
	\|A^{1/2}\vec{u}\|_{H}^{2} 
	\geq \lambda_{1}\norm{\vec{u}}_{H}^{2} \label{Poincare}
\end{equation}
where $\lambda_{1}$ is the first eigenvalue of the Laplacian.
In particular, the $V$-norm is equivalent to the $H^{1}$-norm on $V$.
\end{lemma}
\noindent It is worth mentioning that $\lambda_{1}$ is the first eigenvalue of the scalar Laplacian on the sphere (\cite{Ilin}, p.567).


\begin{lemma}\label{trilinearestimateslemma} For $\vec{u},\vec{v},\vec{w} \in \ V$, the trilinear form satisfies
\begin{align}
	& b(\vec{u},\vec{v},\vec{v})= 0, \label{bproperties} \\
	& \left|b(\vec{u},\vec{v},\vec{w})\right|
	\leq
		k\norm{\vec{u}}^{1/2}_{H}\norm{\vec{u}}^{1/2}_{H^{1}}
		\norm{\vec{v}}_{H^{1}}\norm{\vec{w}}^{1/2}_{H}
		\norm{\vec{w}}^{1/2}_{H^{1}}. \label{bestimate2}
\end{align}
If $\vec{v}\in H^{2}\cap V$ then
\begin{align}
	&  b(\vec{v},\vec{v},A\vec{v})=0,     \label{bproperty} \\
	& \left|b(\vec{u},\vec{v},\vec{w})\right|
	\leq
		k\norm{\vec{u}}_{H}^{1/2}\norm{\vec{u}}_{H^{1}}^{1/2}
		\norm{\vec{v}}^{1/2}_{H^{1}}
		\norm{\vec{v}}^{1/2}_{H^{2}}\norm{\vec{w}}_{H}.  \label{bestimate4}  
\end{align}  
\end{lemma}

For Lemma \ref{whenhaveexponentialconvergence}, it will be necessary to have estimates involving terms of the form $\left\langle \curl_{n}v\times u,w\right\rangle$.

\begin{lemma}\label{trilinearestimateslemma2} For $\vec{u},\vec{v},\vec{w} \in \ V$, the following hold: 
For any $u,v\in H^{1}$
\begin{equation}
\left\langle \curl_{n}u\times v ,v\right\rangle =0.  \label{strongerthanskiba}
\end{equation}

If $v\in H^{2}\cap V$ and $u$ is a zonal vector field (only latitudinal dependence) then
\begin{align}
	& \left\langle \curl_{n}u\times u,v\right\rangle =0 \label{strongerthanskiba2} \\
	& \left\langle \curl_{n}v\times u,Av\right\rangle =0. \label{strongerthanskiba3}
\end{align}

Furthermore, if $\vec{u}=g(t)\curl \sin(\phi)\vec{n}$ ($g(t)$ is uniformly bounded) and $\vec{v}\in H^{2}\cap V$ then
\begin{align}
	 & \left\langle \curl_{n}v\times u,v\right\rangle =0 \label{Calc3} \\  
	 & \left\langle \curl_{n}u\times v,Av\right\rangle =0. \label{Calc4}
	 \end{align}                               
\end{lemma}

\begin{proof}
The proof of \eqref{strongerthanskiba3} and \eqref{Calc4} can be found on pages 69-70 of \cite{Ilin2} (the proof of \eqref{strongerthanskiba3} holds for any zonal force) since the time-dependent function pulls out of the estimate. It remains to prove \eqref{strongerthanskiba}, \eqref{strongerthanskiba2} and \eqref{Calc3}. The arguments are similar to calculations for \eqref{strongerthanskiba3} and \eqref{Calc4}.

Since the sphere is simply connected, for a divergence-free vector field $\vec{u}$, there is a flow function $\psi$ (\cite{Ilin}, pp. 567-568) with mean zero
\begin{equation} \nonumber
		 \vec{u}= -\curl \psi\vec{n} = \vec{n}\times \nabla \psi, \quad
		\curl_{n}\vec{u} = \Delta \psi \vec{n},
\end{equation} 
\noindent where $\Delta$ is the spherical Laplacian for functions.

\noindent For the following calculation, we will need the following information about the spherical Jacobian (\cite{Ilin2}, p.51). Let $\vec{u}= -\curl \bar{\psi}\vec{n}$ and $\vec{v}= -\curl \psi\vec{n}$.
\begin{align}
	& J(a,b)=-\curl_{n}\left(\vec{n}a\times \left(\vec{n}a\times\nabla b\right)\right) = \frac{1}{\cos\phi}\left(\frac{\partial a}{\partial \lambda}\frac{\partial b}{\partial \phi}-\frac{\partial b}{\partial \lambda}\frac{\partial a}{\partial \phi}\right), \nonumber \\
	& \curl_{n}\left(\curl_{n}\vec{v}\times \vec{u}\right)=-J(\Delta \psi,\bar{\psi}), \label{formingJacobian}\\
	& \int_{M}J(a,b)dS^{2}=0, \quad \quad \text{by \ Stoke's \ Theorem}.
\end{align}

Notice that if $\vec{u}$ is zonal, then $\curl_{n}\left(\curl_{n}\vec{u},\vec{u}\right)=-J(\bar{\psi},\bar{\psi})=0$, which establishes \eqref{strongerthanskiba2}.

The following calculation is sufficient to establish \eqref{strongerthanskiba}.
\begin{equation}
	\begin{split}
		\left\langle \curl_{n}\vec{u}\times \vec{v},\vec{v} \right\rangle 
		& = -\left\langle \curl_{n}\vec{u}\times \vec{v},\curl \psi \right\rangle \\
		& = -\left\langle \curl_{n}\left(\vec{n}\Delta \bar{\psi}\times\vec{v}\right), \psi \right\rangle \\
		& =  \int J(\Delta \bar{\psi}, \psi)\psi dS^{2} \\
		& =  \frac{1}{2}\int J(\Delta \bar{\psi}, \psi^{2}) dS^{2} =0.
	\end{split}
\end{equation}

Finally, suppose that $\vec{u}$ is zonal of the form $g(t)\curl \sin(\phi)\vec{n}$. Denoting $\vec{v}=-\curl\psi$, then
\begin{equation}
\begin{split}
  \left\langle \curl_{n}\vec{v}\times \vec{u},\vec{v} \right\rangle
  & = -\left\langle \curl_{n}\vec{v}\times \vec{u},\curl \psi \right\rangle \\
  & = -\left\langle \curl_{n}\left(\vec{n}\Delta \psi\times\vec{u}\right), \psi \right\rangle \\
  & =  \int J(\Delta \psi, g(t)\sin(\phi))\psi dS^{2} \\
  & = g(t)\int_{M} (\partial_{\lambda}\Delta\psi) \psi dS^{2} \\
	& \text{using \ Integration \ by \ Parts} \\
	& = -g(t)\int_{M} (\Delta\psi) \partial_{\lambda}\psi dS^{2}.
\end{split}
\end{equation}
The remainder of the calculation is identical to the calculation on p. 62 of \cite{Ilin2}, which establishes \eqref{Calc3}.
\end{proof}


The following lemma will allow for the Coriolis term $C(\vec{u})$ to vanish from all the estimates. Its proof only uses that the Laplacian commutes with differentiability in the longitudinal direction - see \cite{Skiba1}, p. 635.

\begin{lemma}\label{lemmaforzonalestimates} 
For smooth vector fields $\vec{u}$, the following holds for $r\geq 0$
\begin{equation}
	\left\langle C(\vec{u}),A^{r}\vec{u}\right\rangle_{H}
	=\left\langle l\,\vec{n}\times \vec{u}, A^{r}\vec{u}\right\rangle= 0.  \label{Skiba-estimate}
\end{equation}
\end{lemma}

\noindent We now turn to the proofs of Condition \ref{generalcontraction} and Lemma \ref{whenhaveexponentialconvergence}. Since most of the inequalities in Condition \ref{generalcontraction} are well-known, the necessary estimates will be given without proof. We instead provide the proof for inequality \eqref{Condition2} and Lemma \ref{whenhaveexponentialconvergence}. Recall that $\vec{f}\in L^{\infty}(0,\infty;H)$.

\subsection{Proof of Inequality \eqref{Condition2}}

Let $S_{t,t_{0}}\vec{u}_{0}$ be the solution of the 2D Navier-Stokes equations with initial condition $\vec{u}_{0}$ at time $t_{0}$ (in particular, $S_{t_{0},t_{0}}u=u_{0}$). The following estimates will be needed (cf. \cite{Babin1, Constantin1, Ilin, Kuksin7, Robinson}).


\begin{lemma}
The following inequalities hold for the deterministic 2D Navier-Stokes equation on the sphere for all $t \geq t_{0}$:
 \begin{align}
         &    \norm{S_{t,t_{0}}\vec{u}_{0}}^{2}_{H} \leq \norm{\vec{u}_{0}}^{2}_{H}e^{-\lambda_{1}\nu (t-t_{0})}
             +
             \dfrac{\norm{\vec{f}}_{L^{\infty}(0,\infty;H)}^{2}}{\nu^{2}\lambda_{1}}
              \left(1-e^{-\nu\lambda_{1}(t-t_{0})}\right)
              \label{Hcontraction}\\
				& \nu\int_{s}^{t+s}\norm{S_{\tau,t_{0}}\vec{u}_{0}}_{H^{1}}^{2}d\tau 
					\leq \norm{S_{s,t_{0}}\vec{u}_{0}}_{H}^{2}+ \dfrac{t}{\nu}\norm{\vec{f}}_{L^{\infty}(0,\infty;H)}^{2}. \label{integralinequality} \\
         & \norm{S_{t,t{0}}\vec{u}_{0}}^{2}_{H^{1}}
             \leq \norm{\vec{u}_{0}}^{2}_{H^{1}}e^{-\lambda_{1}\nu (t-t_{0})}
             +
             \dfrac{\norm{\vec{f}}^{2}_{L^{\infty}(0,\infty;H)}}{\nu^{2}\lambda_{1}}
             \left(1-e^{-\nu\lambda_{1}(t-t_{0})}\right)
             \label{Vcontraction} \\
					& \nu\int_{s}^{t+s}\norm{S_{\tau,t_{0}}\vec{u}_{0}}_{H^{2}}^{2}d\tau 
					\leq \norm{S_{s,t_{0}\vec{u}_{0}}}_{H^{1}}^{2}+ tC(\norm{\vec{f}}_{L^{\infty}(0,\infty;H)}), \label{H2integralbounda} 
     \end{align}
     where $\lambda_{1}$ is the first eigenvalue of the operator $-\Delta$ on functions.

Moreover, for any $t\geq \dfrac{1}{2}+t_{0}$
\begin{equation}
    \norm{S_{t,t_{0}}\vec{u}_{0}}^{2}_{H^{1}} \leq K\norm{\vec{u}_{0}}^{2}_{H}e^{-\nu\lambda_{1}(t-t_{0})}+ C_{1}\norm{\vec{f}}_{L^{\infty}(0,\infty;H)}^{2}. \label{VtoHcontraction}
\end{equation}
\end{lemma}

It should be mentioned that the choice of $\frac{1}{2}$ in equation \eqref{VtoHcontraction} is arbitrary and any number can be used, only changing the constants. Furthermore, combining equations \eqref{VtoHcontraction} and \eqref{H2integralbounda} gives
\begin{equation}
	\nu\int_{s}^{t+s}\norm{S_{\tau,t_{0}}\vec{u}_{0}}_{H^{2}}^{2}d\tau 
					\leq C\norm{\vec{u}_{0}}_{H}^{2}+ tC(\norm{\vec{f}}_{L^{\infty}(0,\infty;H)}) + C(\norm{\vec{f}}_{L^{\infty}(0,\infty;H)}). \label{H2integralbound}
\end{equation}

\noindent Now consider the difference between two solutions $\vec{w}=\vec{u}-\vec{v}$
    \begin{equation}
         S_{t,t_{0}}\vec{u}-S_{t,t_{0}}\vec{v} =
            \dfrac{\partial \vec{w}}{\partial t} +
            \nu A+B(\vec{w},\vec{u})+ B(\vec{v},\vec{w}) + C(\vec{w})=0. \label{differenceofsolutions}
    \end{equation}


\begin{lemma} For all $t\geq 0$ the difference of solutions satisfies
\begin{equation}
    \norm{S_{t,t_{0}}\vec{w}_{0}}_{H}^{2}\leq \norm{\vec{w}_{0}}_{H}^{2}\times exp\left(-\nu\lambda_{1}(t-t_{0})+\frac{k^{2}}{\nu^{3}}\norm{f}_{L^{\infty}(0,\infty;H)}^{2}+\frac{k^{2}}{\nu^{2}}\norm{u_{0}}_{H}^{2}\right). \label{continuityofNS}
\end{equation}
\end{lemma}

\begin{remark}\label{differenceofsolutionssatisfiescontraction} By equation \eqref{continuityofNS} in order to ensure \eqref{FForceCondition} it is sufficient that
\begin{equation} \nonumber
    \norm{\vec{f}}_{L^{\infty}(0,\infty;H)} < \dfrac{\nu^{2}\sqrt{\lambda_{1}}}{k}. 
\end{equation}
\end{remark}


Unlike the previous estimates, we use a less standard approach in the following lemma.

\begin{lemma} Let $\vec{w}=\vec{u}-\vec{v}$ and for any $R>0$ let $\norm{\vec{u}_{0}}_{H}<R$ and $\norm{\vec{v}_{0}}_{H}<R$. The following estimate holds for all $t\geq t_{0}+1$:
\begin{equation}
    \norm{S_{t,t_{0}}\vec{w}_{0}}_{H^{1}} \leq
    C(R,f,t-t_{0})\ \norm{\vec{w}_{0}}_{H} . \label{HboundonVdifference}
\end{equation}
\end{lemma}
\begin{proof}
Taking the inner product equation \eqref{differenceofsolutions} with $\vec{u}$ and using equations \eqref{bproperties} and \eqref{bestimate2} gives
\begin{align}
	& \dfrac{1}{2}\partial_{t}\norm{\vec{w}}_{H}^{2}+\nu\norm{\vec{w}}_{H^{1}}^{2} 
	\leq k\norm{\vec{w}}_{H}\norm{\vec{w}}_{H^{1}}\norm{\vec{u}}_{H^{2}}  \nonumber \\
  & \leq \dfrac{\nu}{2}\norm{\vec{w}}_{H^{1}}^{2} + \dfrac{k^{2}}{2\nu}\norm{\vec{w}}^{2}_{H}\norm{\vec{u}}_{H^{1}}^{2} \nonumber \\
	\Rightarrow & \partial_{t}\norm{\vec{w}}_{H}^{2}+\nu\norm{\vec{w}}_{H^{1}}^{2} 
	\leq \dfrac{k^{2}}{\nu}\norm{\vec{w}}^{2}_{H}\norm{\vec{u}}_{H^{1}}^{2}. \label{forVnorm}
\end{align}
Integrating equation \eqref{forVnorm} and using the Mean Value Theorem give that there is an $s\in \left(t_{0}+\frac{1}{2},t_{0}+1\right)$ such that
\begin{align}
        \nu\norm{S_{s,t_{0}}\vec{w}_{0}}_{H^{1}}^{2} \nonumber
    &    =
        2\nu\int_{t_{0}+1/2}^{t_{0}+1}\norm{S_{\tau}\vec{w}_{0}}_{H^{1}}^{2}d\tau \quad \text{and \ by \ \eqref{forVnorm}}  \nonumber \\
    &    \leq
        C\norm{S_{t_{0}+1/2,t_{0}}\vec{w}_{0}}_{H}^{2}
        +
        C\int_{t_{0}+1/2}^{t_{0}+1}\norm{S_{\tau,t_{0}}\vec{w}_{0}}_{H}^{2}
        \norm{S_{\tau,t_{0}}\vec{u}_{0}}_{H^{1}}^{2}d\tau \nonumber \\
		& 	\text{using \ equations \ \eqref{integralinequality} \ and \ \eqref{continuityofNS}}	\nonumber \\
    &    \leq
        C(R,\norm{\vec{f}}_{L^{\infty}(0,\infty;H)})\norm{\vec{w}_{0}}_{H}^{2}.
        \label{boundfors}
\end{align}
Now taking the $L^{2}$ inner product of equation \eqref{differenceofsolutions} with $A\vec{w}$ gives
\begin{equation}
     \dfrac{1}{2}\partial_{t}\norm{\vec{w}}_{H^{1}}^{2}
     +
     \nu\norm{\vec{w}}_{H^{2}}^{2}
        \leq \left|b(\vec{u},\vec{w},A\vec{w})\right| + \left|b(\vec{w},\vec{v},A\vec{w})\right|.
        \label{boundedinnextstep}
\end{equation}
By equation \eqref{bestimate4}, the right side of equation \eqref{boundedinnextstep} is bounded above by
\begin{equation}
\begin{split}
    & \leq k\norm{\vec{u}}_{H^{1}}^{1/2}\norm{\vec{u}}_{H}^{1/2}\norm{\vec{w}}_{H^{1}}^{1/2}
        \norm{\vec{w}}_{H^{2}}^{1/2}\norm{\vec{w}}_{H^{2}}
        +  k\norm{\vec{w}}_{H^{1}}\norm{\vec{v}}_{H^{2}}\norm{\vec{w}}_{H^{2}} \\
    & \leq
        K\norm{\vec{u}}^{2}_{H^{1}}\norm{\vec{u}}_{H}^{2}\norm{\vec{w}}^{2}_{H^{1}}
        + \dfrac{\nu}{2}\norm{\vec{w}}_{H^{2}}^{2}
        + C\norm{\vec{w}}_{H^{1}}^{2}\norm{\vec{v}}_{H^{2}}^{2}\quad    \text{by \ Cauchy}.
\end{split}
\end{equation}
Therefore
\begin{equation}
     \partial_{t}\norm{\vec{w}}_{H^{1}}^{2}
        \leq
        \left(-\nu\lambda_{1}+K\left(\norm{\vec{u}}_{H}^{2}\norm{\vec{u}}_{H^{1}}^{2}+
        \norm{\vec{v}}_{H^{2}}^{2}\right)\right)\norm{\vec{w}}_{H^{1}}^{2}
\end{equation}
and
\begin{eqnarray}
    & \norm{S_{t,t_{0}}\vec{w}_{0}}_{H^{1}}^{2} \leq \norm{S_{s,t_{0}}\vec{w}_{0}}_{H^{1}}^{2}\times \nonumber \\
    & \text{exp}\left(-\nu\lambda_{1}(t-s) + k\int_{s}^{t}
 \left(\norm{S_{\tau}\vec{u}_{0}}_{H^{1}}^{2}\norm{S_{\tau}\vec{u}_{0}}_{H}^{2}
        +\norm{S_{\tau}\vec{v}_{0}}_{H^{2}}^{2}\right)d\tau\right).
\end{eqnarray}
By equations \eqref{Hcontraction} and \eqref{H2integralbound} this is bounded above by (increasing the integrals)
\begin{eqnarray}
        \leq & 
            \norm{S_{s,t_{0}}\vec{w}_{0}}_{H^{1}}^{2}\text{exp}\left(-\nu\lambda_{1}(t-s)\right)\times \\
        &    \text{exp}\left(C(R,f)\int_{t_{0}}^{t}
            \norm{S_{\tau,t_{0}}\vec{u}_{0}}_{H^{1}}^{2}d\tau 
						+C(R)+(t-t_{0})C\norm{\vec{f}}_{L^{\infty}(0,\infty;H)}^{2}\right). \nonumber
\end{eqnarray}
By equations \eqref{integralinequality} and \eqref{boundfors} this is bounded above by
\begin{eqnarray}
  \leq & 
         C(R,f)\norm{\vec{w}_{0}}_{H}^{2} \text{exp}\left(-\nu\lambda_{1}(t-(t_{0}+1)\right)
         \times \nonumber \\
       &     \text{exp}\left(C(R,f)
            +C(R)+(t-t_{0})C\norm{\vec{f}}_{L^{\infty}(0,\infty;H)}^{2}\right)
\end{eqnarray}
which establishes \eqref{HboundonVdifference}.
\end{proof}

\noindent We know turn to the proof of inequality \eqref{Condition2}.

Let $Q= (I-P_{N})S_{t}\vec{w}_{0} = (I-P_{N})(S_{t,t_{0}}(\vec{u}_{0}-\vec{v}_{0}))$.

\noindent Then
\begin{eqnarray}
        &    \norm{Q}_{H}^{2} 
            \leq \dfrac{1}{\lambda_{N+1}}\norm{Q}_{H^{1}}^{2}
            \leq \dfrac{1}{\lambda_{N+1}}\norm{S_{t}\vec{w}_{0}}_{H^{1}}^{2} \nonumber \\
        &    \leq \dfrac{C(R,f,t-t_{0})}{\lambda_{N+1}}\norm{\vec{w}_{0}}_{H}^{2}
            :=\gamma_{N}\norm{\vec{w}_{0}}_{H}^{2},
\end{eqnarray}
where the last step is by \eqref{HboundonVdifference}. For any $t\geq t_{0}+1$ a $N$ can be found (depending on $t-t_{0}$, R, and $\vec{f}$) such that the $\gamma_{N}$ is less than or equal to any $q>0$. Since $t-t_{0}=1$ for the kicked equations, $N$ can be chosen only depending on $R$ and $\vec{f}$.

\subsection{Proof of Lemma \ref{whenhaveexponentialconvergence}}

Since equation \eqref{thisishowsmall} is a well known results (see also Remark \ref{differenceofsolutionssatisfiescontraction}), it remains to establish that there is a unique globally attracting solution under the remaining conditions. The follow lemmas will establish this.

\begin{lemma} If the force generates a solution of the form $g(t)\curl \sin(\phi)\vec{n}$ then the solution is globally attracting.
\end{lemma}

\begin{proof}
The proof is analogous to a calculation in \cite{Ilin2}, pp. 69-70 (done for $\vec{f}=2\nu\curl(-a\sin\phi)$).

\noindent Let $\vec{u}=\bar{\vec{u}}+\vec{u}'$ solve the time-dependent Navier-Stokes equations with forcing $\vec{f}$ where $\vec{u}'$ is a perturbation and $\bar{\vec{u}}$ is the zonal solution $g(t)\curl_{n}\sin(\phi)\vec{n}$. 
The perturbation solves
\begin{equation}
	\partial_{t}\vec{u}'+\nu A\vec{u}' + G\vec{u}' + B(\vec{u}',\vec{u}')=0,
\end{equation}
where 
\begin{equation}
	G\vec{u}'= C(\vec{u}') + \curl_{n}\bar{\vec{u}}\times\vec{u}'+ \curl_{n}\vec{u}'\times\bar{\vec{u}}.
\end{equation}
Dropping the primes for ease of notation and taking the inner product with $\vec{u}$ gives
\begin{equation}
	\dfrac{1}{2}\partial_{t}\norm{\vec{u}}_{H}^{2}
	+
	\nu \norm{\vec{u}}_{V}^{2}
	+ 
	\left\langle G\vec{u},\vec{u}\right\rangle
	=
	0.
\end{equation}
$\left\langle G\vec{u},\vec{u}\right\rangle=0$ by equations \eqref{strongerthanskiba}, \eqref{Calc3}, and \eqref{Skiba-estimate}.
Thus, by equation \eqref{Poincare}, for any $t\geq 0$ the perturbation satisfies
\begin{equation}
	\begin{split}
	& \dfrac{1}{2}\partial_{t}\norm{\vec{u}}_{H}^{2}\leq -\nu\lambda_{1}\norm{\vec{u}}_{H}^{2} \\
	\Rightarrow \
	& \norm{\vec{u(t)}}_{H}^{2} \leq \norm{\vec{u(0)}}_{H}^{2}e^{-2\nu\lambda_{1}t}
	\end{split}
\end{equation}
and the solution is asymptotically attracting in $H$.
\end{proof}

\begin{lemma}
Let $f$ be a force that generates a solution to the Navier-Stokes equations of the form $g(t)\curl \sin(\phi)\vec{n}$. Then there exists $\delta>0$ such that all $g\in L^{\infty}(0,\infty,H)$ such that $\norm{f-g}_{H}<\delta$ generate a unique globally attracting solution.
\end{lemma}

\begin{proof}
The proof will first show that if the solution to the Navier-Stokes equations with a nonzonal force is ``close enough'' to the zonal solution, then it is globally exponentially stable. Standard estimates are then used to express the inequalities in terms of the distance from the force $f$. 

Let $\vec{u}$ be the unique zonal solution of the form $g(t)\curl \sin(\phi)\vec{n}$ for the Navier-Stokes equations with force $\vec{f}$. Suppose $\vec{g}$ is such that there exists $\vec{v}=\vec{u}+\bar{\vec{v}}$ that solves
\begin{equation}
	\partial_{t}\vec{v}+\nu A\vec{v} +B(\vec{v},\vec{v}) + C(\vec{v}) = \vec{g}.  \label{almostzonalnavierstokes}
\end{equation}
Let $\vec{\psi}$ be another solution to \eqref{almostzonalnavierstokes} and consider $\vec{q}= \vec{\psi} - \vec{v}$ which solves
\begin{equation}
	\partial_{t}\vec{q} + \nu A\vec{q} + B(\vec{\psi},\vec{\psi}) - B(\vec{v},\vec{v}) + C(\vec{q}) = 0.
\end{equation}
Rewriting the nonlinear terms gives
\begin{equation}
	\partial_{t}\vec{q} + \nu A\vec{q} + B(\vec{q},\vec{q}) + \curl_{n}\vec{q}\times\vec{u} +\curl_{n}\vec{q}\times\bar{\vec{v}} +\curl_{n}\vec{v}\times\vec{q}  +  C(\vec{q}) = 0.
\end{equation}
Taking the inner product with $\vec{q}$ and using \eqref{bproperties}, \eqref{strongerthanskiba}, \eqref{Calc3} and \eqref{Skiba-estimate} this simplifies to 
\begin{equation}
	\frac{1}{2}\partial_{t}\norm{\vec{q}}_{H}^{2} + \nu \norm{\vec{q}}_{H^{1}}^{2}+\left\langle \curl_{n}\bar{\vec{q}}\times\vec{v} ,\vec{q}\right\rangle = 0.
\end{equation}
Since $\left\langle \curl_{n}\vec{v}\times\vec{u},\vec{w}\right\rangle$ satisfies analogous inequalities to $b(\vec{u},\vec{v},\vec{w})$ including \eqref{bestimate2}
\begin{equation}
	\frac{1}{2}\partial_{t}\norm{\vec{q}}_{H}^{2} + \nu \norm{\vec{q}}_{H^{1}}^{2} \leq  C\norm{\vec{q}}_{H^{1}}^{3/2}\norm{\bar{\vec{v}}}_{H^{1}}\norm{\vec{q}}_{H}^{1/2}.
\end{equation}

Cauchy's Inequality gives that
\begin{equation}
\begin{split}
	& \frac{1}{2}\partial_{t}\norm{\vec{q}}_{H}^{2} + \nu \norm{\vec{q}}_{H^{1}}^{2} \leq  \frac{\nu}{2}\norm{\vec{q}}_{H^{1}}^{2}+C\norm{\bar{\vec{v}}}_{H^{1}}^{4}\norm{\vec{q}}_{H}^{2} \\
\Rightarrow	&	\partial_{t}\norm{\vec{q}}_{H}^{2}
			\leq
			\norm{\vec{q}}_{H}^{2}
			\left(-\lambda_{1}\nu 
			+C\norm{\bar{\vec{v}}}_{H^{1}}^{4}\right) \\
\Rightarrow & \norm{\vec{q}}_{H}^{2}
\leq \norm{\vec{q_{0}}}_{H}^{2}exp\left(-\lambda_{1}\nu t + C\int_{0}^{t}\norm{\overline{\vec{v(\tau)}}}_{H^{1}}^{4}d\tau\right). \label{almostzonalrestriction}	
			\end{split}
\end{equation}
Thus if $\norm{\bar{\vec{v}}}_{H^{1}}$ is small enough, there will be a unique globally attracting solution in $H$.
It remains to express the norms of $\bar{\vec{v}}$ in terms of the difference of forces. Since $\bar{\vec{v}}=\vec{v}-\vec{u}$, consider the difference between the Navier-Stokes equations with force $\vec{f}$ and zonal solution $\vec{u}=g(t)\curl \sin(\phi)\vec{n}$ and equation \eqref{almostzonalnavierstokes} getting 
\begin{equation}
	\partial_{t}\bar{\vec{v}} +\nu A\bar{\vec{v}} - B(\vec{u},\vec{u})+B(\vec{v},\vec{v})+C(\bar{\vec{v}})=\vec{f}-\vec{g}.
\end{equation}
Since 
$- B(\vec{u},\vec{u})+B(\vec{v},\vec{v})= \vec{u}\times \curl_{n}\bar{\vec{v}} + \bar{\vec{v}}\times \curl_{n}\vec{u} + B(\bar{\vec{v}},\bar{\vec{v}})$, the inner product with $A\bar{\vec{v}}$ and equations \eqref{bproperty}, \eqref{strongerthanskiba3}, \eqref{Calc4}, and \eqref{Skiba-estimate} give
\begin{equation}
	\partial_{t}\norm{\bar{\vec{v}}}_{H^{1}}^{2} +\nu \norm{\bar{\vec{v}}}_{H^{2}}^{2}\leq C\norm{\vec{f}-\vec{g}}_{H}^{2}. \label{forwhatscoming}
\end{equation}
Using equation \eqref{Poincare}, equation integrating \eqref{forwhatscoming} from $\left[\frac{1}{2},t\right]$ and using the fact that $\norm{\bar{\vec{v}}(1/2)}_{H^{1}}\leq \norm{\vec{u}(1/2)}_{H^{1}}+\norm{\vec{v}(1/2)}_{H^{1}}$ and \eqref{VtoHcontraction} yields
\begin{equation}
\begin{split}
	\norm{\bar{\vec{v}}(t)}_{H^{1}}^{2} & 
	\leq \norm{\bar{\vec{v}}(1/2)}_{H^{1}}^{2}e^{-\lambda_{1}\nu (t-1/2)} + C\norm{\vec{f}-\vec{g}}_{L^{\infty}(0,\infty;H)}^{2} \\
	&	\leq C(\norm{\vec{v}_{0}}_{H},\norm{\vec{u}_{0}}_{H})e^{-\lambda_{1}\nu t}+C\norm{\vec{f}-\vec{g}}_{L^{\infty}(0,\infty;H)}^{2}.
	\end{split}
\end{equation}
Thus by Cauchy's inequality
\begin{equation}
	\norm{\bar{\vec{v}}(t)}^{4}_{H^{1}}\leq C(\norm{\vec{v}_{0}}_{H},\norm{\vec{u}_{0}}_{H})e^{-2\lambda_{1}\nu t}+C\norm{\vec{f}-\vec{g}}_{L^{\infty}(0,\infty;H)}^{4}.
\end{equation}
Thus the term in the exponential in \eqref{almostzonalrestriction} is bounded above by 
\begin{equation} \nonumber
C(\norm{\vec{v}_{0}}_{H},\norm{\vec{u}_{0}}_{H}) + t\left(-\lambda_{1}\nu+C\norm{\vec{f}-\vec{g}}_{L^{\infty}(0,\infty;H)}^{4}\right). 
\end{equation}
Therefore there is $\delta>0$ such that if $\norm{\vec{f}-\vec{g}}_{L^{\infty}(0,\infty;H)}\leq \delta$ then the unique solution $\vec{v}$ is globally exponentially stable in $H$.
\end{proof}

%

\begin{thebibliography}{}
%
\bibitem[1]{Babin1} Babin, A.; Vishik, M.: {\it Attractor of Evolutionary Equation}. Studies in Mathematics and its Applications. Vol. {\bf 25}, Amsterdam, North-Holland. 1992.

\bibitem[2]{Brzezniak} Brzezniak, Z.; Goldys, B.; Le Gia, Q. Random Dynamical Systems Generated by Stochastic Navier-Stokes Equation on the Rotating Sphere. Journal of Mathematical Analysis and Applications, 426, no. 1, 505-545 (2015)
	
\bibitem[3]{Constantin1} Constantin, P.; Foias, P.: {\it Navier-Stokes Equations}. Chicago Lecture in Mathematics, Chicago-London: University of Chicago Press, 1988	
	
  
  
	\bibitem[4]{Heywood} Heywood, J.; Rannacher, R. An analysis of stability concepts for the Navier-Stokes equations. J. Reine Angew. Math. 372, 133. (1986)
		
	\bibitem[5]{Ilin} Il'in, A. The Navier-Stokes and Euler equations on two-dimensional closed manifolds. Math USSR-Sb., {\bf 69}, no 2, 559-579. (1991)
  
  	\bibitem[6]{Ilin2} Il'in, A. Partly dissipative semigroups generated by the Navier-Stokes system on two-dimensional manifolds, and their attractors. Russian Acad. Sci. Sb. Math., {\bf 78}, no 1, 159-182. (1994)
  
	
	\bibitem[7]{Ilin5} Ilyin, A. Stability and Instability of Generalized Kolmogorov Flows on the Two-Dimensional Sphere. Adv. in Diff. Eq. {\bf 9}, 979-1008. (2004)
 
	\bibitem[8]{Kuksin} Kuksin, S.; Shirikyan, A. A coupling approach to randomly forced nonlinear PDE's. I, Comm. Math. Phys. 221, 351-366. (2001)
	
	\bibitem[9]{Kuksin2} Kuksin, S.; Shirikyan, A. Stochastic dissipative PDE's and Gibbs measures. Comm. Math. Phys. 213, 291-330. (2000)
	
	\bibitem[10]{Kuksin3} Kuksin, S.; Shirikyan, A. {\it Randomly Forced Nonlinear PDEs and Statistical Hydrodynamics in 2 Space Dimensions}. European Mathematical Society. (2006)
	
	\bibitem[11]{Kuksin4} Kuksin, S.; Shirikyan, A. Coupling approach to white-forced nonlinear PDE's. J. Math. Pures Appl. 81, no. 6, 567-602. (2002)
	
	\bibitem[12]{Kuksin7} Kuksin, S.; Shirikyan, A. {\it Mathematics of Two-Dimensional Turbulence}. Cambridge University Press. (2012)

	\bibitem[13]{Kuksin8} Kuksin, S.; Piatnitski, A.; Shirikyan, A. A Coupling Approach to Randomly Forced Nonlinear PDE's. II. Comm. Math. Phys. 230, 81-856. (2002)
  
 	
 	\bibitem[14]{Robinson} Robinson, J., {\it Infinite-Dimensional Dynamical Systems}, Cambridge University Press. (2001)
    
 	\bibitem[15]{Shirikyan} Shirikyan, A. Ergodicity for a class of Markov processes and applications to randomly forced PDE's. I., Russ. J. Math. Phys. 12, no. 1, 81-96. (2005)
    
 	\bibitem[16]{Shirikyan4} Shirikyan, A. Control and Mixing for 2D Navier-Stokes equations with space-time localised noise. Annales Scientifiques de l'ENS. 48, no. 2, 253-280. (2015)
  	
	\bibitem[17]{Skiba1} Skiba, Y. On the existence and uniqueness of solution to problems of fluid dynamics on a sphere. J. Math. Anal. Appl. 388, no. 1, 627-644. (2012)

	\bibitem[18]{Titi} Cao, C.; Rammaha, M.; Titi, E.. The Navier-Stokes equations on the rotating 2-D sphere: Gevrey regularity and asymptotic degrees of freedom.  Z. Angew. Math. Phys. 50, 341-360 (1999)
  
 	\bibitem[19]{Varner} Varner, G. Stochastically perturbed Navier-Stokes system on the rotating sphere. Diss. University of Missouri, Columbia, MO. May 2013.

	\bibitem[20]{Varner2} Varner, G. (2015) Unique Measure for the Time-Periodic Navier-Stokes on the Sphere. Applied Mathematics, {\bf 6}, 1809-1830. http://dx.doi.org/10.4236/am.2015.611160 
\end{thebibliography}
%

\end{document}